\documentclass{amsart}

\usepackage[utf8]{inputenc}
\usepackage{amsmath, amsfonts, amsthm, amssymb,mathrsfs,color}
\usepackage[margin=1.5in]{geometry}
\usepackage{graphicx}

\usepackage[dvipsnames]{xcolor} 

\usepackage[hidelinks]{hyperref}
\hypersetup{linktocpage,bookmarksopen = true, bookmarksopenlevel = 1, colorlinks=true,
linkcolor = Bittersweet, citecolor = RoyalBlue , urlcolor = gray, pdfstartview = FitR}
\usepackage{caption, subcaption} 
\usepackage{soul}
\usepackage{array}
\usepackage{amsmath,amsthm,amssymb,amscd,url,enumerate, mathrsfs}
\usepackage{mathtools}

\usepackage{booktabs}
\usepackage{multirow}
\usepackage{upgreek}

\usepackage{afterpage}
\usepackage{float} 
\usepackage{enumitem}

\newcounter{tablerow}
\newcommand{\rowtag}[2]{%
  \refstepcounter{tablerow}%
  \label{#1}%
  \makebox[2em][r]{(#2)}%
}

\definecolor{myorange}{rgb}{0.9, 0.55, 0.3}
\definecolor{mygreen}{rgb}{0.35, 0.71, 0.0}
\definecolor{mybrown}{rgb}{0.63, 0.32, 0.18}

\theoremstyle{plain} 
\newtheorem{theorem}{Theorem}[section]

\newtheorem{proposition}[theorem]{Proposition}
\newtheorem{lemma}[theorem]{Lemma}

\theoremstyle{definition}

\theoremstyle{remark}
\newtheorem{remark}[theorem]{Remark}
\newtheorem{correction}[theorem]{Correction}

\numberwithin{equation}{section}

\newcounter{rowcntr}[table]
\renewcommand{\therowcntr}{\thetable.\arabic{rowcntr}}

\newcolumntype{N}{>{\refstepcounter{rowcntr}\therowcntr}c}

  {\end{list}}

  {\end{list}}

\newcommand{\tightoverset}[2]{%
  \mathop{#2}\limits^{\vbox to -.5ex{\kern-1.05ex\hbox{$#1$}\vss}}}

\newcommand{\ZZ}{\mathbb{Z}}

\newcommand{\setS}{{\mathcal{S}}}

\renewcommand{\gcd}{{\operatorname{gcd}}}

\newcommand{\MOD}[1]{~(\textup{mod}~#1)}
\renewcommand{\pmod}{\MOD}

\newcommand{\mtil}{\overline{m}}


\thanks{
 Fuchs was supported by NSF DMS-2154624. 
 Hsu was supported by an
AMS-Simons Research Enhancement Grants for PUI Faculty. 
Rickards was supported by NSERC Discovery Grant RGPIN-2025-04068. 
Schindler was supported by NWO 016.Veni.173.016.
Stange was supported by NSF DMS-2401580.}

\title[Primes in shifted quadratic forms: primitivity and congruence]{Primes represented by shifted quadratic forms:  on primitivity and congruence classes}

\author[Fuchs]{Elena Fuchs}
\address{University of California Davis, Davis, California, USA}
\email{efuchs@ucdavis.edu}
\urladdr{https://www.math.ucdavis.edu/~efuchs/}

\author[Hsu]{Catherine Hsu}
\address{Swarthmore College, Swarthmore, Pennsylvania, USA}
\email{chsu2@swarthmore.edu}
\urladdr{https://chsu.domains.swarthmore.edu/}

\author[Rickards]{James Rickards}
\address{Saint Mary's University, Halifax, Nova Scotia, Canada}
\email{james.rickards@smu.ca}
\urladdr{https://jamesrickards-canada.github.io/}

\author[Schindler]{Damaris Schindler}
\address{Göttingen University, Göttingen, Germany}
\email{damaris.schindler@mathematik.uni-goettingen.de}
\urladdr{https://sites.google.com/site/damarishomepage/}

\author[Stange]{Katherine E. Stange}
\address{University of Colorado Boulder, Boulder, Colorado, USA}
\email{kstange@math.colorado.edu}
\urladdr{https://math.katestange.net/}

\date{\today}
\begin{document}

\begin{abstract}
We prove lower bounds of the form $\gg N/(\log N)^{3/2}$ for the number of primes up to $N$ \emph{primitively} represented by a shifted positive definite integral binary quadratic form, and under the additional condition that primes are from an arithmetic progression.  This extends the sieve methods of Iwaniec, who showed such lower bounds without the primitivity and congruence conditions.  Imposing primitivity adds some subtleties to the local criteria for representation of a shifted prime: for example, some shifted quadratic forms of discriminant $5 \pmod{8}$ \emph{do not} primitively represent infinitely many primes. We also provide a careful list of the local conditions under which a genus of an integral binary quadratic form represents an integer, verified by computer, and correcting some minor errors in previous statements.  The motivation for this work is as a tool for the study of prime components in Apollonian circle packings \cite{primecomponents}.
\end{abstract}

\maketitle

\section{Introduction}

We prove the following theorem, which extends work of Iwaniec \cite{Iwaniec72shiftedprimes} on primes represented by shifted binary quadratic forms.  We consider lower bounds for the number of such primes when we impose two simultaneous restrictions: that the representation be \emph{primitive} (sometimes called \emph{proper}); and that the primes fall in a specified congruence class.  

\begin{theorem}
\label{thm:maintheorem}
Let $f=ax^2+bxy+cy^2$ be a primitive integral binary quadratic form with discriminant $D$ not a perfect square, positive definite if $D<0$, and $(a,2D)=1$.  Let $A$ be a non-zero integer and $B\in \mathbb{N}$ with $\gcd(A,B)=1$.  Denote by $b_{f,A}(n)$ the characteristic function for whether $n$ is represented primitively by $Bf(x,y)+A$. Assume that $2\mid AB$ or $D\not\equiv 5 \pmod{8}$. 
Let $\mtil$ be a natural number and $l$ an integer with $\gcd(l,\mtil)=1$ and assume that $\gcd(\mtil,2DB)=1$ and that $\gcd(l-A,\mtil)=1$. 
Then 
\begin{enumerate}
\item\[
\sum_{\substack{p \le N,\\ p \text{ prime}}} b_{f,A}(p)
\gg \ll
\frac{N}{(\log N)^{3/2}}.
\]
\item
\[
\sum_{\substack{p \le N,\\ p \text{ prime}\\p\equiv \ell\bmod{\mtil}}} b_{f,A}(p)
\gg
\frac{N}{(\log N)^{3/2}}.
\]
\end{enumerate}

The notation above is taken from \cite{Iwaniec72shiftedprimes}, where $D=b^2-4ac=\pm 2^{\vartheta_2}p_1^{\vartheta_{p_1}}\cdots p_k^{\vartheta_{p_k}}$ and  $D_p=p^{-\vartheta_p}D$.
\end{theorem}

Note that the upper bound is immediate from \cite{Iwaniec72shiftedprimes}, in which Iwaniec shows the desired upper and lower bounds for primes represented (not necessarily primitively) by $Bf(x,y)+A$. 

Iwaniec's celebrated work on primes represented by quadratic polynomials in two variables includes \cite{Iwaniec72a, Iwaniec72shiftedprimes, Iwaniec74quadraticfunctions}.  The last of these is the most general, addressing quadratic polynomials with possible linear factors.  This extends earlier work of Pleasants showing the representation of infinitely many primes \cite{Pleasants}.  In \cite{Iwaniec72shiftedprimes}, Iwaniec introduces the half-dimensional sieve.

More recently, these results have been featured in applications to Apollonian circle packings in \cite{Sarnak}, \cite{BourgainFuchs}, and \cite{bourgain2012representation}, as outlined in Section~\ref{sec:apollonian}.

\subsection{Methods}
\label{sec:methods}

In this paper, we extend the methods of Iwaniec to prove Theorem~\ref{thm:maintheorem}.  The proofs follow the arguments of \cite{Iwaniec72shiftedprimes} quite closely, so that we have chosen to preserve the notation of that paper in order to facilitate reading them side-by-side.

In \cite{Sarnak}, Sarnak remarks that the more general result in \cite{Iwaniec74quadraticfunctions} (and implicitly those in \cite{Iwaniec72shiftedprimes}) can be modified to include the condition of primitive representation.  We found this modification is not entirely straightforward. Indeed, the statement is false for certain choices of $A, B,$ and $D$.

For example, if $D \equiv 5 \mod{8}$ and $2\nmid AB$, then the shifted quadratic form $B(ax^2+bxy+cy^2)+A$ does not primitively represent infinitely many primes. Indeed, if it were to primitively represent infinitely many primes, then $ax^2+bxy+cy^2$ would primitively represent infinitely often even numbers. However, this is not possible by Theorem \ref{thm:Iwaniecthm2}. As a concrete example, $x^2-xy+y^2+1$ does not primitively represent infinitely many primes.

We also make use of methods from \cite{BourgainFuchs}, which enable us to pass from squarefree shifted primes represented by the genus of a quadratic form to those represented by one quadratic form. Note that in \cite{Iwaniec72shiftedprimes}, Iwaniec also has a method for passing from the genus of a quadratic form to the one quadratic form of interest, but in doing so he sacrifices primitivity. Our strategy to focus on (mostly) squarefree shifted primes represented by the genus first requires a careful analysis of local conditions on the various parameters introduced in \cite{Iwaniec72shiftedprimes}. Our application of \cite{BourgainFuchs} is the first example in the literature of such an application, to our knowledge.

\subsection{Local conditions for representation by a genus}

It is well-known that representability by a genus is a local property: one need only verify a condition for each prime $p$ (including the infinite prime).  In most cases, the condition is simply the Kronecker symbol $\left( \frac{D}{p} \right) = 1$.  However, primes dividing the discriminant, and the prime $2$, have more complex conditions.  In \cite{Iwaniec72shiftedprimes}, Iwaniec provides exhaustive tables of these local conditions (see Tables ~\ref{table1} and \ref{table2} in Section~\ref{sec:representability}).  This plays a role in the proof there, and it plays a larger role in our proof here.  We are unaware of any other exhaustive statement in the literature.

In this paper we correct a small notational error in the original tables and verify the statements by computer experiment.

\subsection{Application to Apollonian packings.}\label{sec:apollonian}

The study of Apollonian circle packings has attracted attention in recent years \cite{FuchsBulletin}.  An Apollonian circle packing (Figure~\ref{fig:pack}) gives rise to a collection of curvatures (curvature is the inverse of radius), which can be viewed as the union of the primitively represented values of a collection of translated binary quadratic forms.  Any four mutually tangent circles, or their quadruple of curvatures, is called a \emph{Descartes quadruple}.

\begin{figure}[h]
  \begin{center}
  \includegraphics[width=5in]{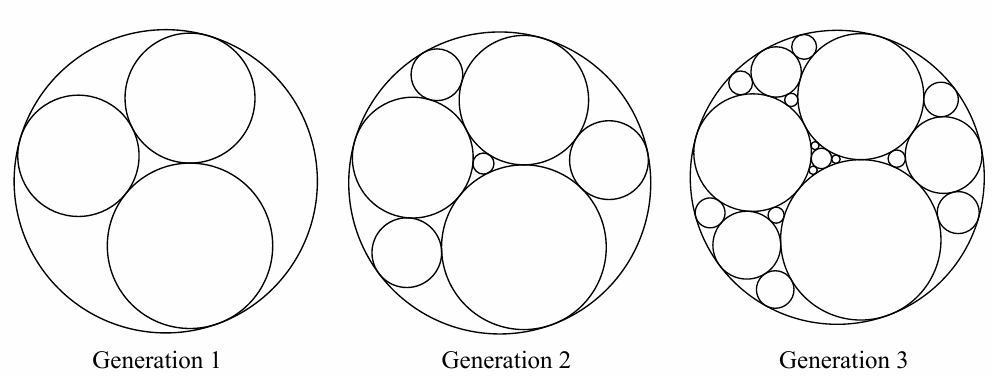}
  \caption{The construction of an Apollonian packing.  Begin with four mutually tangent circles.  For each triple of mutually tangent circles, a result of Apollonius asserts the existence of exactly two which complete the triple to a mutually tangent quadruple.  Add these circles in.  For all newly created mutually tangent triples, add in their completions.  Repeat, ad infinitum.}
  \label{fig:pack}
  \end{center}
\end{figure}

\begin{proposition}[{\cite{Sarnak}}]
\label{prop:QF}
    Let $C_a$ be a circle in an Apollonian circle packing, having curvature $a$ and lying within a Descartes quadruple with curvatures $(a,b,c,d)$.  Let $S$ be the set of circles tangent to $C_a$ within the packing.  Define
   \[
     f_a(x,y) - a = (b+a)x^2+(a+b+d-c)xy+(d+a)y^2 - a.
    \]
    Then, the multiset of curvatures of the circles in $S$ is exactly the multiset of values $f_a(x,y)-a$ for $(x,y)$ coprime integers.
\end{proposition}

The form $f_a(x,y)$ has discriminant $-4a^2$, and it is known that all primitive integral binary quadratic forms of discriminant $-4a^2$ arise from some Apollonian circle packing in this way \cite{GLMWY}.

In \cite{primecomponents}, the authors consider the question of \emph{prime components} in Apollonian circle packings, that is, maximal tangency-connected collections of circles of prime curvature.  This requires control over the prime translated primitively represented values of such forms, in congruence classes.  The version of Theorem~\ref{thm:maintheorem} that is needed in this application is the following.

\begin{theorem}
\label{thm:maintheorempaper2}
Let $f_a$ be a primitive positive definite integral binary quadratic form of the form given in Proposition~\ref{prop:QF}, in particular, having discriminant $D = -4 a^2$. Assume also that $a$ is odd.  Write $b_{f_a,a}(n)$ for the characteristic function for whether $n$ is represented primitively by $f_a(x,y)-a$.  
\begin{enumerate}
\item  Then \[
\sum_{\substack{p \le N,\\ p \text{ prime}}} b_{f_a,a}(p)
\gg \ll
\frac{N}{(\log N)^{3/2}}.
\]

\item Let $\ell$ and $m$ be coprime integers satisfying $(m,2D)=1$ and $(\ell + a, m) = 1$.  Then
\[
\sum_{\substack{p \le N,\\ p \text{ prime}\\p\equiv \ell\bmod{m}}} b_{f_a,a}(p)
\gg_m 
\frac{N}{(\log N)^{3/2}},
\]
where the implicit constant may depend on $m$. 
\end{enumerate}
\end{theorem}

Part (1) of Theorem~1.3 has been accepted as fact among experts on Apollonian packings, likely because it was assumed that the generalization of Iwaniec's result presented in this paper is immediate. Indeed, Corollary~1.3 of \cite{bourgain2012representation} assumes a result we prove here---Proposition~\ref{prop:lowerboundSsq}---and uses it in conjunction with the main results of that paper, so it is fitting that it is this main result of \cite{bourgain2012representation} that enables us to prove Theorem~1.3 above and hence close the gap in Corollary~1.3 of that paper.\footnote{Note that in \cite{bourgain2012representation}, the discriminant is allowed to be as large as a power of $\log N$, which is not something we tackle here. To apply our result to such a large discriminant situation would require additional work.}

\subsection{Paper structure}  This paper is organized as follows: in Section~\ref{sec:representability}, we recall in detail Iwaniec's results on representability of a number by the genus of a quadratic form. Section~\ref{sec:proof} we carry out the sieve arguments to essentially obtain a lower bound on the number of squarefree shifted primes represented by the genus of our quadratic form, and then apply \cite{bourgain2012representation} to pass to a single form. Section~\ref{sec:local} addresses necessary local conditions to make Theorem~\ref{thm:maintheorem} hold. Scattered throughout the paper, the reader will find small typographical corrections to \cite{Iwaniec72shiftedprimes}.  None of these are consequential. 

\vspace{0.1in}

\noindent {\bf Acknowledgements:} We thank Henryk Iwaniec for a conversation which inspired the methods in this paper.  We also thank the organizers of the conference \emph{Women in Numbers 4}, which led to this work.

\section{Representability by the genus}
\label{sec:representability}

Here we record Iwaniec's original statements that we modify, generalize, or use; and set the notation to match the notation used in that paper, with one notable exception:  we use $f$, $g$ for binary quadratic forms (instead of $\upvarphi_0$ and $\upvarphi$).

For an integer $a$, write $a^*$ for the largest positive odd divisor.  
 We use $\left( \frac{a}{b} \right)$ to denote the Kronecker symbol.  

The quadratic forms we consider have the form\footnote{Iwaniec's condition that $a > 0$ is merely a convenience to force definite forms to be positive definite.}
\[
f( \xi, \eta) = a\xi^2 + b \xi \eta + c \eta^2, \quad a > 0, \quad (a,b,c)=1,
\]
with
\[
D  =b^2 - 4ac
\]
not a perfect square.  
Without loss of generality (by an equivalence of quadratic forms) we may assume $(a,2D)=1$.  
We will factor $D$ as
\begin{equation}
    \label{eqn:Ddecomp}
    D = \pm 2^{\vartheta_2} p_1^{\vartheta_{p_1}} \cdots p_r^{\vartheta_{p_r}}, \quad \vartheta_{p_i} \ge 1, \quad i=1,\ldots, r, \quad \vartheta_2 \ge 0.
\end{equation}
Let $D_p = p^{-\vartheta_p} D$.  Then $D_2 = \pm D^*$.
Let $R_f$ denote the genus of $f$.

Suppose we wish to represent $n \in \ZZ$.  Iwaniec assumes\footnote{If one wishes to consider negative integers, one uses the same conditions, together with a condition ``at infinity'', i.e. a negative number cannot be represented by a positive definite form.} $n > 0$. We write $n$ with reference to the primes of $D$, as follows:  
\begin{equation}\label{eq:dm}
n = 2^{\varepsilon_2} p_1^{\varepsilon_{p_1}} \cdots p_r^{\varepsilon_{p_r}} m = dm, \quad \varepsilon_{p_i} \ge 0, \quad (m, 2D) = 1.
\end{equation}
Let $d_p = p^{-\varepsilon_p}d$. 
Then $d_2 = d^*$.

\begin{correction}
    In \cite{Iwaniec72shiftedprimes}, $d_p$ is erroneously defined as $p^{-\vartheta_p}d$ instead of $p^{-\varepsilon_p}d$.
\end{correction}

Given a genus $R_f$, we say than an integer $N$ is \emph{table-admissible} for $R_f$ if for each $p \mid N$, there exists a row in Tables~\ref{table1} and \ref{table2}, specifically one of rows~(\hyperref[row1]{1})--(\hyperref[row5]{5}) if $p = p_i \neq 2$;
        row (\hyperref[row6]{6}) if $p \nmid 2D$; and one of rows ~(\hyperref[row7]{7})--(\hyperref[row20]{20}) if $p=2$, such that
        \begin{enumerate}[label=\upshape{(\arabic*)}]
        \item $\varepsilon_{p}$ and $\vartheta_{p}$ satisfy the conditions in the first column;
            \item $m \in \mathscr{L}(\varepsilon_{p}, \vartheta_{p})$ of the fourth column; and
            \item $D$ satisfies the conditions of the fifth column.
        \end{enumerate}

\begin{theorem}[{\cite[Theorem 2]{Iwaniec72shiftedprimes}}]
\label{thm:Iwaniecthm2}
        The integer $n$ is primitively representable by $R_f$ if and only if $2nD$ is table-admissible for $R_f$.
\end{theorem}

\begin{table}[H]
\begin{minipage}{\textwidth}
{
\renewcommand{\arraystretch}{1.0}
\begin{tabular}{ccccccc} \toprule
     \multicolumn{2}{c}{$\varepsilon_{p_i}$ and $\vartheta_{p_i}$} & {$\aleph$} & {$\tau$} & {$\mathscr{L}(\varepsilon_{p_i},\vartheta_{p_i}) \subseteq (\ZZ/\tau\ZZ)$} & {Conditions on $D$} & \multicolumn{1}{c}{} \\ 
     \midrule
     \multicolumn{2}{c}{$\varepsilon_{p_i} < \vartheta_{p_i}$, $2 \mid \varepsilon_{p_i}$}
     & $\frac{p_i-1}{2}$ & $p_i$ & $\mathscr{L}'_{p_i} := \left\{ t : \left(\frac{t}{p_i}\right) = \left(\frac{ad_{p_i}}{p_i}\right)\right\}$& none &  \rowtag{row1}{1} \\ 
     \hline
          \multirow{5}{4em}{$\varepsilon_{p_i} = \vartheta_{p_i}$} & \multirow{4}{2.8em}{$2 \mid \varepsilon_{p_i}$}
     & \multirow{2}{1em}{\centering $1$} & \multirow{2}{1em}{\centering $1$} & \multirow{2}{2em}{\centering $\{0\}$} & \multirow{2}{9em}{\centering $p_i > 3$ or\\ ($p_i = 3$, $3 \mid 1+D_3$)} &  \rowtag{row2}{2} \\ 
     & \\ \cline{3-7}
                & 
     & \multirow{2}{1em}{\centering $1$} & \multirow{2}{1em}{\centering $3$}& \multirow{2}{10em}{\centering $\mathscr{L}''_3 := \{-ad_3 \}$} & \multirow{2}{9em}{\centering $p_i = 3$, $D_3 \equiv 1 \pmod{3}$} &  \rowtag{row3}{3} \\ 
     & \\ \cline{2-7}
                & $2 \nmid \varepsilon_{p_i}$
     & $\frac{p_i-1}{2}$ & $p_i$ & $\mathscr{L}''_{p_i} := \left\{ t : \left(\frac{t}{p_i}\right) = \left(\frac{-ad_{p_i}D_{p_i}}{p_i}\right) \right\}$ & none &  \rowtag{row4}{4} \\ \hline
          \multicolumn{2}{c}{$\varepsilon_{p_i} > \vartheta_{p_i}$, $2 \mid \vartheta_{p_i}$}
     & $1$ & $1$ & $\{0\}$& $\left( \frac{D_{p_i}}{p_i} \right)=1$ &  \rowtag{row5}{5} \\ \hline
          \multicolumn{2}{c}{$p \mid m$, $\vartheta_{p} = 0$}
     & $1$ & $1$ & $\{0\}$ & $\left( \frac{D}{p} \right)=1$  &  \rowtag{row6}{6} \\ 
 \bottomrule
\end{tabular}
}
\end{minipage}
\caption{Table 1 from \cite{Iwaniec72shiftedprimes}, corrected.  This table applies for odd primes.  Each row of the table represents a set of conditions on the prime (denoted $p_i$ if it is one of the odd primes dividing $D$, i.e., of the decomposition \eqref{eqn:Ddecomp}; and $p$ if it is coprime to $2D$), given by the first, fourth and fifth columns, as described in Theorem~\ref{thm:Iwaniecthm2}.  The other two columns provide extra information: the column headed $\aleph$ lists the cardinality of the set $\mathscr{L}(\epsilon_p, \upvartheta_p)$, and the column $\tau$ denotes the modulus with respect to which the elements of $\mathscr{L}(\epsilon_p, \upvartheta_p)$ are to be interpreted.  The last column is merely a labeling of the rows for reference.}
\label{table1}
\end{table}

\begin{minipage}{\textwidth}
{
\renewcommand{\arraystretch}{1.25}
\begin{tabular}{ccccccc} \toprule
     \multicolumn{2}{c}{$\varepsilon_2$ and $\vartheta_2$} & {$\aleph$} & {$\tau$} & {$\mathscr{L}(\varepsilon_2,\vartheta_2) \subseteq (\ZZ/\tau\ZZ)$} & {Conditions on $D$} & \\ \midrule
          \multirow{5}{4em}{$\varepsilon_{2} = 0$} & {$\vartheta_{2}=0$}
     & $1$ & $1$ & $\{0\}$ & $D \equiv  1 \pmod{4}$ & \rowtag{row7}{7} \\ \cline{2-7}
               & {$\vartheta_{2}=2$}
     & $1$ & $4$ & $\{ad_2\}$ & $D_2 \equiv -1 \pmod{4}$ &  \rowtag{row8a}{8a} \\\cline{2-7}
                    & {$\vartheta_{2}=2$}
     & $1$ & $1$ & $\{0\}$ & $D_2 \equiv 1 \pmod{4}$ &  \rowtag{row8b}{8b} \\ \cline{2-7}
                  & {$\vartheta_{2}=3$}
     & $2$ & $8$ & $\{ad_2, ad_2(1-2D_2) \}$ & none &  \rowtag{row9}{9} \\ \cline{2-7}
                       & {$\vartheta_{2}=4$}
     & $1$ & $4$ & $\{ad_2 \}$ & none &  \rowtag{row10}{10} \\ \cline{2-7}
                       & {$\vartheta_{2}\ge 5$}
     & $1$ & $8$ & $\{ad_2\}$ & none &  \rowtag{row11}{11} \\ \cline{1-7}
                      $\varepsilon_2 \ge 1$     & {$\vartheta_{2} =0$}
     & $1$ & $1$ & \{0\} & $D \equiv 1 \pmod{8}$ &  \rowtag{row12}{12} \\ \cline{1-7}
     \multirow{8}{4em}{$\varepsilon_{2} \ge 1$, $\vartheta_2 \ge 1$} & {$2 \mid \varepsilon_{2} \le \vartheta_2 - 5$}
     & $1$ & $8$ & $\{ ad_2 \}$ & none &  \rowtag{row13}{13} \\ \cline{2-7}
           & {$2 \mid \varepsilon_{2} = \vartheta_2 - 4$}
     & $1$ & $8$ & $\{ 5ad_2 \}$ & none &  \rowtag{row14}{14} \\ \cline{2-7}
              & {$2 \mid \varepsilon_{2} = \vartheta_2 - 3$}
     & $1$ & $8$ & $\{ ad_2(1-2D_2) \}$ & none &  \rowtag{row15}{15} \\ \cline{2-7}
                & {$2 \mid \varepsilon_{2} = \vartheta_2 - 2$}
     & $1$ & $4$ & $\{-ad_2D_2 \}$ & none &  \rowtag{row16}{16} \\ \cline{2-7}\cline{2-7}
                     & {$2 \nmid\varepsilon_{2} = \vartheta_2 - 2$}
     & $2$ & $8$ & $\{ -ad_2D_2, ad_2(2-D_2) \}$ & none &  \rowtag{row17}{17} \\ \cline{2-7}
                          & {$2 \nmid \varepsilon_{2} = \vartheta_2 - 1$}
     & $1$ & $4$ & $\{ ad_2\frac{1-D_2}{2}\}$ & $D_2 \equiv -1 \pmod{4}$ &  \rowtag{row18}{18} \\ \cline{2-7}
                 & {$2 \mid \varepsilon_{2} = \vartheta_2$}
     & $1$ & $1$ & $\{0\}$ & $D_2 \equiv 5 \pmod{8}$ &  \rowtag{row19}{19} \\ \cline{2-7}
                    & {$2 \mid\vartheta_2, \varepsilon_{2} > \vartheta_2$}
     & $1$ & $1$ & $\{0\}$ &  $D_2 \equiv 1 \pmod{8}$ &  \rowtag{row20}{20} \\ 
 \bottomrule
\end{tabular}\captionof{table}{Table 2 from \cite{Iwaniec72shiftedprimes}, corrected, dedicated to the prime $p=2$. The method of reading the table is the same as for Table~\ref{table1}.}
\label{table2}}
\end{minipage}

\begin{correction}
   In Tables~\ref{table1} and \ref{table2}, a correction has been made to \cite{Iwaniec72shiftedprimes}. Instead of $D^*$ in the original, one should use $D_2$ as in our tables. The replacement of $D^*$ by $D_2$ allows the table to be additionally valid for negative integers $n$ (subject to an additional condition at $\infty$). For clarity purposes, line (8) has been broken into two cases. The theorem was clarified to note that only prime divisors of $2nD$ require a line in the table. All other changes are cosmetic.
\end{correction}

The tables specify sets $\mathscr{L}(\varepsilon_p, \vartheta_p)$ of `possible $m$' depending on $a$, $D$ and $d$.  (Iwaniec only used this notation on the case of $p=2$ but we have defined it more generally.)  The column $\aleph$ gives $\# \mathscr{L}(\varepsilon_p, \vartheta_p)$.

Next, following Iwaniec, we give an alternate statement that is more useful for application. 
We write $k(n)$ for the squarefree kernel of $n$, i.e. $n = k(n) n_0^2$ where $n_0^2$ is the largest square divisor.  In particular, the Kronecker symbol is given by a quadratic character $\chi_{\mathfrak{f}(a)}$ of conductor\footnote{Iwaniec used $f$ for the conductor in place of $\mathfrak{f}$ but as we wish to reserve $f$ for a binary quadratic form, we have changed the font.} $\mathfrak{f}(a)$:
\[
\left( \frac{a}{b} \right)
= \chi_{\mathfrak{f}(a)}(b), \quad
\mathfrak{f}(a) = \left\{
\begin{array}{ll} 
k(a) & k(a) \equiv 1 \pmod{4} \\
4k(a) & k(a) \not\equiv 1 \pmod{4} \\
\end{array}\right.\!\!.
\]
Let 
\[
P = \left\{
p : \left( \frac{k(D)}{p} \right) = 1
\right\}.
\]

With reference to $\tau$ as a function of $\varepsilon_{p}$ and $\vartheta_p$ given in Tables~\ref{table1} and \ref{table2}, define
\begin{equation}
    \label{eqn:Q}
Q = \tau( \varepsilon_2, \vartheta_2) \prod_{p \mid D^*} \tau(\varepsilon_p, \vartheta_p) 
\end{equation}

\begin{theorem}[{\cite[Theorem 3]{Iwaniec72shiftedprimes}}]
\label{thm:I3}
    Let $n = dm$ as in (\ref{eq:dm}).  Then $n$ is primitively representable by $R_f$ if and only if \begin{enumerate}
    \item $d$ is table-admissible for $R_f$;
         \item $m \equiv L \pmod{Q}$ for some $0 < L < Q$ satisfying 
    \begin{enumerate}
        \item\label{item2a} $L \equiv \ell \pmod \tau$ for some $\ell \in \mathscr{L}(\varepsilon_2, \vartheta_2)$;
        \item for each $p_i \mid Q^*$ having $\varepsilon_{p_i} < \vartheta_{p_i}$, there exists $\ell \in \mathscr{L}'_{p_i}$ such that $L \equiv \ell \pmod{p_i}$;
                \item\label{item2c} for each $p_i \mid Q^*$ having $\varepsilon_{p_i} = \vartheta_{p_i}$, there exists $\ell \in \mathscr{L}''_{p_i}$ such that $L \equiv \ell \pmod{p_i}$;           
    \end{enumerate}   
        \item all prime factors of $m$ belong to $P$.  
    \end{enumerate}
\end{theorem}

Residues $L$ modulo $Q$ satisfying \eqref{item2a}--\eqref{item2c} will be called \emph{residue-admissible}.

In the proof of Theorem~\ref{thm:maintheorem}, care will be given to checking that suitable values $d$ and $L$ actually exist. This is not guaranteed and is outlined in Lemma~\ref{lemlocal}.

\begin{remark}
We have computationally verified the correctness of the updated Tables \ref{table1} and \ref{table2} for a range of discriminants $D$ and integers $n$. Using PARI/GP \cite{PARI2}, we can initialize quadratic forms, find a representative for each element of the class group, and separate them into genera. Using the \texttt{qfbsolve} command, we can find all proper representations of an integer by a quadratic form, and compare these results to Tables \ref{table1} and \ref{table2}. The \texttt{qfbsolve} command uses Gauss reduction and quadratic form equivalence testing to determine all solutions.

For all discriminants $D$ with $-1000\leq D\leq 1000$ and non-zero integers $n$ with $-100,000\leq n\leq 100,000$, we verified Theorem~\ref{thm:Iwaniecthm2}.
This code can be found in the GitHub repository \cite{GHBQFRepCheck}. In particular, \texttt{\detokenize{test_thm2(D, n1, n2)}} checks that the theorem is correct for all genera of discriminant \texttt{D} and all non-zero integers between \texttt{n1} and \texttt{n2}. Running the code on the above range took around 30 minutes on a laptop.

\end{remark}

\section{Proof of Theorem~\ref{thm:maintheorem}}\label{sec:proof}

Here we prove the lower bound in Theorem~\ref{thm:maintheorem}, as the upper bound is a direct consequence of \cite{Iwaniec72shiftedprimes}. 
We use $\mu^2(x)=1$ to denote the condition that $x$ is squarefree (this is the square of the M\"obius function). For a finite set of primes $\setS$ and $x\in \mathbb{N}$ we write 
\begin{equation*}
\mu^2_{\setS} (x)=\left\{\begin{array}{ccc} 1 & \mbox{ if } & x \mbox{ is squarefree outside of } \setS\\
0 & & \mbox{ otherwise.} \\ \end{array}\right.
\end{equation*}
For the rest of this section we set $\setS=\{2,3\}$.

\subsection{Set-up}

We define the counting function $S_{sq}(N)$ to be the number of primes $p\leq N$ which satisfy the conditions
\begin{eqnarray*}
&\mu_{\setS}^2(B^{-1}(p-A))=1,\ p-A=Bg(x,y) \mbox{ for some } g\in R_f \mbox{ and } x,y\in \mathbb{Z},\ \gcd(x,y)=1\\& 2^{5}\nmid B^{-1}(p-A),\ 3^3\nmid B^{-1}(p-A), \quad p\equiv l\mod \mtil,
\end{eqnarray*}
where $R_f$ is the genus of $f$.
A central step towards proving Theorem 1 will be to provide a lower bound on this quantity. 

Let $Z>3$ be a parameter to be chosen later. Define
\begin{equation}\label{eq:Pi}
\Pi_{Z,D,A}=\prod_{\substack{p<Z\\ p\nmid 2DA\mtil}}p.
\end{equation}
Let $S_1(N,Z)$ be equal to the number of primes $p\leq N$ with $p\equiv l\mod \mtil$ such that 
\begin{eqnarray*}&&\gcd\left(\frac{p-A}{B},\Pi_{Z,D,A}\right)=1,\\&&\frac{p-A}{B}\mbox{ is primitively represented by some } g \in R_f,\\&& q^2\nmid \frac{p-A}{B} ,\, \forall q\mid D \mbox{ with } (q,6)=1,\\
&& 2^5\nmid \frac{p-A}{B} , 3^3 \nmid \frac{p-A}{B}.
\end{eqnarray*}
Moreover for every $d\in \mathbb{N}$ we define $S_2(N,d)$ to be equal to the number of primes $p\leq N$ such that $p\equiv l\mod \mtil$ and
$$\frac{p-A}{B}=d^2g(x,y) \mbox{ for some form } g\in R_f,\ \gcd(x,y)=1.$$

Eventually, we want a lower bound on $S_{sq}(N)$, as any squarefree number represented by a form is by default primitively represented. We will then use result from \cite{bourgain2012representation} to move from the genus to one form. Our first step towards this is to break the task into two individual counting functions.  The following Lemma makes use of Theorem~\ref{thm:I3}.

\begin{lemma}
\label{lemma:sq}
$$S_{sq}(N) \geq S_1(N,Z) - \sum_{Z\leq d\ll N^{1/2}} S_2(N,d) - O(1),$$
where the implicit constant depends only on $A$.
\end{lemma}

\begin{proof} 
We will show that if $n = \frac{p -A}{B}$ is counted by $S_1(N,Z)$ but not by $S_{sq}(N)$, then it arises as an element of one of the counting functions $S_2(N,d)$.

Let $n = \frac{p -A}{B}$, $p$ prime with $p\equiv l\mod \mtil$, such that $\gcd(\frac{p-A}{B},\Pi_{Z,D,A})=1$ and such that $n = \frac{p-A}{B}$ is primitively represented by some form in the genus of $f$. Moreover assume that if $q\mid D$ and $(q,6)=1$, then $q^2\nmid \frac{p-A}{B}$, and we assume that $2^5\nmid \frac{p-A}{B}$ and $3^3 \nmid \frac{p-A}{B}$. If we assume that $n = \frac{p-A}{B}$ is not counted by $S_{sq}(N)$ and $p>A$, then there exists a prime $q \geq Z$ such that 
$$q^2\mid \frac{p-A}{B}.$$ 
Write $n = \frac{p-A}{B} = q^2q_0$ for some $q_0 \in \ZZ$. From here on we follow the notation of Iwaniec's article \cite{Iwaniec72shiftedprimes} and Section~\ref{sec:representability}.

Note that $Q$ (as in Section~\ref{sec:representability}) is supported on primes dividing $2D$ by construction (in fact $Q \mid 8 D$). 
So if we choose $Z>2$ sufficiently large, then $q$ is coprime to $Q$.

Write $q^2 q_0=dm$ such that $d$ is composed of primes dividing $2D$ and such that $\gcd(m,2D)=1$.  Observe that $\gcd(d,q) = 1$, so $q^2 \mid m$.  Then, by Theorem~\ref{thm:I3}, the number $d$ is table-admissible, and we have
$$p\mid m\Rightarrow p\in P.$$
Also, there exists a residue-admissible residue class $L$ modulo $Q$ such that 
$$m\equiv L \mod Q.$$
 We now use Theorem~\ref{thm:I3}
 to show that $q_0 = dm/q^2 = dm'$ 
 is primitively represented by the genus
 $R_f$.   We verify each condition in turn. 
 First, $d$ is the same as before, so it continues to satisfy the conditions of the tables.  Next, all primes diving $m'$ are contained in $P$. 
 
 Third, 
 \[
 m' = m/q^2 \equiv \bar q^2 L \mod Q,
 \] 
where we write $\bar q$ for the inverse of $q$ modulo $Q$.  We wish to show that, if $L$ is residue-admissible according to the second part of Theorem~\ref{thm:I3}, then so is $\bar q^2 L$.  Note that $\bar q^2 \equiv 1 \mod 8$ and hence the residue class $\bar q ^2 L$ remains admissible at the place $2$. For all other primes $p_i$ dividing $D$, we note that the sets $\mathscr{L}_{p_i}'$ and $\mathscr{L}_{p_i}''$ are invariant under multiplication with invertible squares, so admissibility is maintained here also.

Thus, we apply Theorem~\ref{thm:I3} and deduce that the number $q_0$ is primitively represented by the genus of $f$. This finishes the proof of the claim.
\end{proof}

In Subsection~\ref{sec:squarefree-sieve}, we give an upper bound on the second term; and in Subsection~\ref{sec:half-dimensional-sieve}, we give a lower bound on the first term.

\subsection{The squarefree sieve}
\label{sec:squarefree-sieve}

We begin with an upper bound on the $S_2(N,d)$ term in Lemma~\ref{lemma:sq}, making use of Theorem 1 in \cite{Iwaniec72shiftedprimes}.

\begin{lemma}
\label{lemma:sigma4}
For every $\epsilon >0$ have
$$\sum_{Z\leq d\ll N^{1/2}} S_2(N,d) \ll  \frac{N}{(\log N)^{3/2} Z^{1-\epsilon}} + \frac{N}{(\log N)^2}. $$
\end{lemma}

\begin{proof}
If $\gcd(d,A)>1$, then $S_2(N,d)\leq 1$. So we now consider the case $\gcd(d,A)=1$. An application of Theorem 1 in \cite{Iwaniec72shiftedprimes} gives the upper bound
$$S_2(N,d)\ll \frac{N}{\phi(d^2)(\log N)^{3/2}} + \frac{N}{(\log N)^5}$$
and note that the implicit constant is independent of $d$.

Using the crude lower bound $\phi(n)\gg n^{1-\epsilon}$, we obtain the bound
\begin{align*}
\sum_{Z\leq d\leq (\log N)^{3}} S_2(N,d) &\ll \sum_{Z\leq d\leq (\log N)^{3}}\left( \frac{N}{\phi(d^2)(\log N)^{3/2}} + \frac{N}{(\log N)^5}\right) \\
&\ll \frac{N}{(\log N)^{3/2}}\sum_{Z\leq d\leq (\log N)^{3}} \frac{1}{d^{2-\epsilon}}  + \frac{N}{(\log N)^2}\\
&\ll \frac{N}{(\log N)^{3/2} Z^{1-\epsilon}} + \frac{N}{(\log N)^2}.
\end{align*}
We bound the contribution for $ (\log N)^3<d \ll N^{1/2}$ by simply observing that $S_2(N,d)\ll \frac{N}{d^2}.$ With this we obtain
\begin{align*}
\sum_{(\log N)^{3}\leq d\ll N^{1/2}} S_2(N,d) \ll \sum_{(\log N)^{3}\leq d\ll N^{1/2}} \frac{N}{d^2} \ll \frac{N}{(\log N)^3}.
\end{align*}
We deduce that
$$ \sum_{Z\leq d\ll N^{1/2}} S_2(N,d) \ll  \frac{N}{(\log N)^{3/2} Z^{1-\epsilon}} + \frac{N}{(\log N)^2}. $$
\end{proof}

\subsection{Setting up a $1/2$-dimensional sieve}
\label{sec:half-dimensional-sieve}

In this section, we turn to lower bounds for the counting function $S_1(N,Z)$.  Since we are modifying arguments of Iwaniec \cite{Iwaniec72shiftedprimes}, we match the notation of that paper, to ease comparison.

We will now define a set $\mathscr{M}$ over which we sieve, i.e. we will reduce the problem to a sum suitable for the $1/2$-dimensional sieve, of the form
\[
 \sum_{\substack{m\in \mathscr{M}\\ q\mid m \Rightarrow q\in P}}1.
\]
The corresponding part of \cite{Iwaniec72shiftedprimes} is pp. 214-5.

We define the counting function
\[
S_{1,Z}(N,f,B,A) = \sum_{\substack{p\leq N\\ p=Bg(\xi,\eta)+A\\ (\xi,\eta)=1,\ g\in R_{f}\\ (\frac{p-A}{B},\Pi_{Z,D,A})=1\\ p\equiv l \mod \mtil\\q^2\nmid \frac{p-A}{B} ,\, \forall q\mid D,\ (q,6)=1\\ 2^5\nmid \frac{p-A}{B},\ 3^3\nmid \frac{p-A}{B}}}1,
\]

where $\Pi_{Z,D,A}$ is as in (\ref{eq:Pi}). This is the same as Iwaniec's $S_1(N, f,B,A)$ with the addition of a condition that $\frac{p-A}{B}$ be coprime to $\Pi_{Z,D,A}$ and $p\equiv l \mod \mtil$ and that $q^2\nmid \frac{p-A}{B}$ for $q\mid D$ with $(q,6)=1$ and $2^5\nmid \frac{p-A}{B},\ 3^3\nmid \frac{p-A}{B}$.
Then we have
\[
S_{1,Z}(N, f, B, A) = S_1(N, Z).
\]

We assume that $\gcd(\mtil,2DB)=1$ and that $\gcd(l-A,\mtil)=1$. We start by supposing a table-admissible integer $d$ and a residue-admissible $0<L<Q$ (Theorem~\ref{thm:I3}) such that $\gcd(BdL + A, QBd)=1$, where $Q$ is defined in \eqref{eqn:Q}.  We demonstrate the existence of such an integer in Lemma~\ref{lemlocal} (Section~\ref{sec:local}); we also see that we can choose such an integer $d$ in the cases that are relevant to us, in a way such that $d$ is squarefree outside of $2,3$ and such that $2^5\nmid d$ and $3^3\nmid d$.

    \begin{correction}
        The issue of existence of admissible $d, L$ is relevant in the original also.  In \cite[p.229, (4.8)]{Iwaniec72shiftedprimes}, the quantity $\Omega_{A,B,D}$ is defined.  It is possible for this quantity to be zero. Thus the lower bound on $\Omega_{A,B,D}$ in the middle of page 230 does not always hold. This is only a problem if there are no $d$ and $L$ satisfying the conditions $2 \mid ABd$ and $(BdL + A, QBd) = 1$.  As we will see in Lemma~\ref{lemlocal}, the only case of concern is if $2 \nmid AB$ and $D \equiv 5 \pmod{8}$.  However, there is a way to circumvent this.  For the lower bounds on $S_\infty$ and $S_2$ in Iwaniec's Theorem 1, there is no primitivity condition on the representation, hence on can replace here $B$ by $4B$ by substituting the variables $\xi$ and $\eta$ by $2\xi$ and $2\eta$ and with this obtain forms for which one has the desired lower bound on $\Omega_{A,4B,D}$ (which is used on page 231). 
    \end{correction}

\color{black}
Note that $Q$ is only composed of primes dividing $2D$.
We consider the sets
\[
\mathscr{C} = \{ p : |A| < p \le N, p \equiv BdL + A \pmod{QBd},\ p\equiv l\pmod \mtil\},
\]
and
\[
\mathscr{M}' = \left\{ m' : m' = \frac{p-A}{Bd}, p \in \mathscr{C} \right\},
\]
where $p$ is always understood to be a prime. As $\gcd(\mtil,2BD)=1$ we note that the two congruence conditions in $\mathscr{C}$ are independent and define a unique congruence class for $p$ modulo $\mtil QBd$. 
We define the set $\mathscr{M}$, introducing a coprimality to $\Pi_{Z,D,A}$:
\[
\mathscr{M} = \{ m \in \mathscr{M}' : \gcd(m, 2D\Pi_{Z,D,A} ) = 1 \}.
\]

Let $P$ be as in Theorem \ref{thm:I3}. We rewrite the counting function $S_{1,Z}(N,f,B,A)$ as follows: 
\begin{align*}
S_{1,Z}(N,f,B,A) = \sum_{\substack{p\leq N\\ p=Bg(\xi,\eta)+A\\ (\xi,\eta)=1,\ g\in R_{f}\\ (\frac{p-A}{B},\Pi_{Z,D,A})=1\\ p\equiv l \pmod \mtil\\ q^2\nmid \frac{p-A}{B} ,\, \forall q\mid D,\ (q,6)=1\\ 2^5\nmid \frac{p-A}{B},\ 3^3\nmid \frac{p-A}{B}}}1 = \sum_{\substack{d, L\in \mathscr{L} \\\mu_{\setS}^2(d)=1\\ \text{admissible}\\ 2^5\nmid d,\ 3^3\nmid d}}  \sum_{\substack{|A|< p\leq N\\ p\equiv BdL+A \mod{QBd}\\q\mid \frac{p-A}{Bd}\Rightarrow q\in P\\ (\frac{p-A}{Bd},2D\Pi_{Z,D,A})=1 \\ p\equiv l \pmod \mtil}}1+O(|A|). \\
\end{align*}

Here the sums over $d$ and $L\in \mathscr{L}$ are over all admissible pairs as in Theorem \ref{thm:I3}. Note that if $L,d$ are an admissible pair with $\gcd(QBd,BdL+A)>1$, and $p\equiv BdL+A\pmod{QBd}$, then $p$ must be a divisor of $2D$ (as $\gcd(A,B)=1$) and hence

\begin{align}\label{eqndecompS1}
S_{1,Z}(N,f,B,A)= \sum_{\substack{d\\ 2\mid ABd\\\mu_{\setS}^2(d)=1\\2^5\nmid d,\ 3^3\nmid d }}\sum_{\substack{L\in \mathscr{L}\\\gcd(QBd,BdL+A)=1}} \sum_{\substack{m\in \mathscr{M}\\ q\mid m \Rightarrow q\in P}}1 + O(|2DA|).
\end{align}

Hence we are interested in lower bounds for the number of $m\in \mathscr{M}$ such that $m$ is only divisible by primes in $P$, which leads us to a sieve problem.

\subsection{Level of distribution}\label{sec:dist}
Our next goal is to provide a level of distribution for the set $\mathscr{M}$. For this it is sufficient to consider moduli $\varrho$ with $\gcd(\varrho,2D\Pi_{Z,D,A}\mtil)=1$ as elements $m\in \mathscr{M}$ are automatically coprime to $2D\Pi_{Z,D,A}\mtil$ (using the assumption that $\gcd(l-A,\mtil)=1$). 

Let $\varrho$ be a modulus with $\gcd(\varrho, 2D\Pi_{Z,D,A}\mtil)=1$ and write 
$$\mathscr{M}_\varrho = \{ m\in \mathscr{M}: \varrho \mid m\}.$$
Our goal is to approximately evaluate $|\mathscr{M}_{\varrho}|$ and bound the error term. For this we compare the set $\mathscr{M}$ again with the set $\mathscr{M}'$. Note that if $m\in \mathscr{M}'$, then we have $\gcd(m,QA)=1$. 
For a modulus $\sigma$ with $\gcd(\sigma,QA)=1$ we set
$$\mathscr{M}'_\sigma=\{ m\in \mathscr{M}': \sigma\mid m\}. $$
We define $D_1$ to be the greatest divisor of $2D\Pi_{Z,D,A}$ coprime to $QA$. For a modulus $\varrho$ with $\gcd(\varrho,2D\Pi_{Z,D,A})=1$, we observe that
\begin{equation}\label{eqn1}
|\mathscr{M}_{\varrho}|= \sum_{\varrho_1\mid D_1} |\mathscr{M}'_{\varrho\varrho_1}| \mu(\varrho_1).
\end{equation}
Note that for every divisor $\varrho_1\mid D_1$ we have $\gcd(\varrho,\varrho_1)=1$ and $\gcd(\varrho\varrho_1,QA)=1$.\par
Consider $\sigma$ with $\gcd(\sigma,QA\mtil)=1$. Let $\sigma_0$ be an integer with $\sigma_0Q+L\equiv 0 \pmod \sigma$. Then we have
$$\mathscr{M}'_\sigma = \left\{m:\ m=\frac{p-A}{Bd}, p\equiv A+BdL+QBd\sigma_0 \pmod{QBd\sigma}, \ p\equiv l \pmod{ \mtil},\ |A|<p\leq N\right\}.$$
Under the additional assumption $\gcd(BdL+A,QBd)=1$ we have
$$\gcd(A+BdL+QBd\sigma_0,QBd\sigma)=1$$
and $\gcd(\mtil,QBd\sigma)=1$. Hence the two congruence conditions in $\mathscr{M}_\sigma'$ give a unique invertible congruence class, say $\overline{l}$, for the primes $p$ modulo $QBd\sigma\mtil$ and we obtain (using the standard notation $\pi(N,e,f)$ for the number of primes $p$ less than $N$ satisfying $p \equiv f \pmod{e}$)
$$||\mathscr{M}_\sigma'|- \pi(N,QBd\sigma\mtil,\overline{l})|\leq |A|.$$
Assume we are given a modulus $\varrho$ with $\gcd(\varrho,2D\Pi_{Z,D,A}\mtil)=1$. By equation (\ref{eqn1}) we obtain 
\newcommand{\Li}{\operatorname{Li}}
\begin{align*}
&\left| |\mathscr{M}_{\varrho}| - \Li N \sum_{\varrho_1\mid D_1} \frac{\mu(\varrho_1)}{\varphi(QBd\varrho\varrho_1\mtil)}\right| \\
&\leq 2|D\Pi_{Z,D,A}| \max_{\substack{\overline{l},x\\ |\varrho| \leq x \leq |2D\Pi_{Z,D,A}\varrho|\\ (\overline{l},QBdx\mtil)=1}} \left| \pi(N,QBdx\mtil,\overline{l})- \frac{\Li N}{\varphi(QBdx\mtil)}\right| + 2|DA\Pi_{Z,D,A}|.
\end{align*}
Here we used that $|D_1|\leq |2D\Pi_{Z,D,A}|$.\par

We pause a moment to analyse the leading constant
$\sum_{\varrho_1\mid D_1} \frac{\mu(\varrho_1)}{\varphi(QBd\varrho\varrho_1\mtil)}$. For a positive integer $E$ we define
$$\varphi_E(n)=\frac{\varphi(En)}{\varphi(E)}= n\prod_{\substack{p\mid n\\ p\nmid E}}\left(1-\frac{1}{p}\right)$$
and observe that $\varphi_E(n)$ is a multiplicative function. Let $E= QBd\mtil$. Then for $\gcd(\varrho,2D\Pi_{Z,D,A})=1$ and $\varrho_1\mid D_1$ we have $\gcd(\varrho,\varrho_1)=1$ and
\begin{align*}
\sum_{\varrho_1\mid D_1} \frac{\mu(\varrho_1)}{\varphi(QBd\varrho\varrho_1\mtil)} &= \sum_{\varrho_1\mid D_1} \frac{\mu(\varrho_1)}{\varphi(QBd\mtil) \varphi_{QBd\mtil}(\varrho\varrho_1)} \\&= \frac{1}{\varphi(QBd\mtil) \varphi_{QBd\mtil}(\varrho)}\prod_{p\mid D_1} \left(1-\frac{1}{\varphi_{QBd\mtil}(p)}\right)\\
&=\frac{1}{\varphi(QBd\mtil) \varphi_{QBd\mtil}(\varrho)}\prod_{\substack{p\mid D_1\\p\nmid Bd\mtil}} \left(1-\frac{1}{p-1}\right) \prod_{\substack{p\mid D_1\\p\mid Bd\mtil}} \left(1-\frac{1}{p}\right)\\
&=\frac{1}{\varphi(QBd\mtil) \varphi_{QBd\mtil}(\varrho)}\prod_{\substack{p\mid D_1\\p\nmid Bd\mtil}} \left(1-\frac{1}{(p-1)^2}\right) \prod_{p\mid D_1} \left(1-\frac{1}{p}\right).
\end{align*}

\subsection{Applying the sieve}
The corresponding part of \cite{Iwaniec72shiftedprimes} is Section 4; \cite[Lemmata 4.1-4.3]{Iwaniec72shiftedprimes} are general results which we use in the exact same form.

We define the set $\mathscr{P}$ to be the set of primes not dividing $2DA\Pi_{Z,D,A}$ such that $\left(\frac{k(D)}{p}\right)=-1$.
Let $E=QBd\mtil$. Then we have
\begin{equation}
\left| \sum_{\substack{p\leq x\\ p\in \mathscr{P}}} \frac{\log p}{\varphi_E(p)}- \frac{1}{2}\log x\right| < C_1,\quad x\geq 1,
\end{equation}
for a constant $C_1$ which may depend on $A,D,Z,\mtil$, but does not depend on $B$. We define 
$$Y=\prod_{p\mid D_1}\left(1-\frac{1}{p}\right) \prod_{\substack{p\mid D_1\\ p\nmid Bd\mtil}} \left(1-\frac{1}{(p-1)^2}\right) \Li N$$

Note that $Y$ depends on $D_1$ and $D_1$ depends on $\Pi_{Z,D,A}$. We define $y$ as
$$y=\sqrt{N}/|QBdD\Pi_{Z,D,A}\mtil| (\log N)^U,$$
and assume $|QBd\Pi_{Z,D,A}| < (\log N)^{15}$, where $U$ is given as in Theorem 4.1 in \cite{Iwaniec72shiftedprimes}.\par
For $1\leq s \leq 3$ and $N$ sufficiently large we can hence apply Theorem 5 in \cite{Iwaniec72shiftedprimes} and we obtain the lower bound
\begin{align*}
A(\mathscr{M},y^{1/s}) \geq \frac{\sqrt{2e^{\gamma}\pi^{-1}}C_0}{\varphi(QBd\mtil)} \prod_{p\mid D_1}\left(1-\frac{1}{p}\right) \prod_{\substack{p\mid D_1\\ p\nmid Bd\mtil}} \left(1-\frac{1}{(p-1)^2}\right)\\ \times \frac{N}{(\log N)^{3/2}} \left(\int_1^s \frac{dt}{\sqrt{t(t-1)}}+o(1)\right) +O(N(\log N)^{-20}).
\end{align*}
Here the implicit constants are independent of $B$. Moreover, the constant $C_0$ is given by
$$C_0= e^{-\gamma/2} \prod_{\substack{p\nmid 2DBA\Pi_{Z,D,A}\mtil\\ \left(\frac{k(D)}{p}\right)=-1}} \left(1-\frac{1}{(p-1)^2}\right) \prod_{p\mid 2DA\Pi_{Z,D,A}} \left(1-\frac{1}{p}\right)^{-1/2} \prod_{p\nmid 2DA\Pi_{Z,D,A}} \left(1-\frac{1}{p}\right)^{-\frac{1}{2}\left(\frac{k(D)}{p}\right)}$$
As in Iwaniec's paper \cite{Iwaniec72shiftedprimes} following his Corollary to Theorem 4, we observe that for $1<s<\frac{4}{3}$ and $N$ sufficiently large we have
\begin{equation}\label{splittingA}
A(\mathscr{M},y^{1/s}) = \sum_{\substack{m\in \mathscr{M}\\ q\mid m\Rightarrow q\in P}}1 + \sum_{\substack{p_1p_2m\in \mathscr{M}\\ q\mid m \Rightarrow q\in P\\ y^{1/s}\leq p_1,p_2\in \mathscr{P}}}1.
\end{equation}

An inspection of the proof of Lemma 4.4 in \cite{Iwaniec72shiftedprimes} shows that we have the following lemma.

\begin{lemma}
Assume that $|QBdD\Pi_{Z,D,A}\mtil| < (\log N)^{15}$, $15<U$ and $s>1$. Then we have
$$\sum_{\substack{p_1p_2m\in \mathscr{M}\\ q\mid m \Rightarrow q\in P\\ y^{1/s}\leq p_1,p_2\in \mathscr{P}}}1 < \frac{4e^{\gamma/2}C_0\sqrt{s-1}}{\sqrt{\pi}\varphi(QBd\mtil)\sqrt{s}} \log (2s-1) \frac{4s^2N}{(\log N)^{3/2}} \prod_{p\mid D_1}\left(1-\frac{1}{p}\right) \prod_{\substack{p\mid D_1\\ p\nmid Bd}}\left(1-\frac{1}{(p-1)^2}\right)(1+o(1)),$$
where the implicit constant is independent of $B$.
\end{lemma}

\begin{remark}
In Lemma 4.4 in \cite{Iwaniec72shiftedprimes} there is a typo in the statement of the lemma, the factors $\prod_{p\mid D_1} \left(1-\frac{1}{p}\right) \prod_{\substack{p\mid D_1\\ p\nmid Bd}} \left(1-\frac{1}{(p-1)^2}\right)$ are missing on the right hand side of the inequality. 
\end{remark}

By taking $s$ sufficiently close to $1$ in equation \eqref{splittingA} we obtain the following lemma.

\begin{lemma}\label{lemlowerboundM}
Assume that $|QBdD\Pi_{Z,D,A}\mtil| < (\log N)^{15}$. Then we have
$$\sum_{\substack{m\in \mathscr{M}\\ q\mid m\Rightarrow q\in P}}1\gg C_0\prod_{p\leq Z} \left(1-\frac{1}{p}\right)\frac{1}{\varphi(B)}\frac{N}{(\log N)^{3/2}} - O(N(\log N)^{-20}),$$
where the implicit constants are independent of $B$.
\end{lemma}

As a direct consequence we obtain a lower bound for the counting function $S_1(N,Z)$. The following proposition is a combination of Lemma \ref{lemlowerboundM}, equation \eqref{eqndecompS1} and Lemma \ref{lemlocal}. 

\begin{proposition}
Assume that $|100 ABD^3\Pi_{Z,D,A}\mtil| < (\log N)^{15}$. Moreover, we assume that $2\mid AB$ or $D\not\equiv 5 \mod{8}$.
Then we have
$$S_1(N,Z) \gg C_0\prod_{p\leq Z} \left(1-\frac{1}{p}\right) \frac{1}{\varphi(B)}\frac{N}{(\log N)^{3/2}} - O(N(\log N)^{-20}),$$
where the implicit constants are independent of $B$.
\end{proposition}

\subsection{Passing from the genus to one form}

By choosing $Z$ a sufficiently large constant and combining Lemma \ref{lemma:sq} and Lemma \ref{lemma:sigma4} we obtain the following proposition.

\begin{proposition}\label{prop:lowerboundSsq}
Assume that $\gcd(A,B)=1$. Let $\mtil$ be a natural number and $l$ and integer with $\gcd(l,\mtil)=1$ and assume that $\gcd(\mtil,2DB)=1$ and that $\gcd(l-A,\mtil)=1$. Moreover, we assume that $2\mid AB$ or $D\not\equiv 5 \mod{8}$.
Then we have
$$S_{sq}(N) \gg \frac{N}{(\log N)^{3/2}},$$
where the implicit constant is allowed to depend on all involved parameters.
\end{proposition}

The essential ingredient in passing from representations of $\frac{p-A}{B}$ by some form in the genus of $f$ to representations by $f$ is given by a result of Bourgain and Fuchs on the representation of integers by binary quadratic forms.

The following Theorem is a slight generalization of Theorem 1.2 in \cite{bourgain2012representation} and is proved in the same way.
\begin{theorem}\label{thm:generalizationBourgainFuchs}
Let $f=ax^2+bxy+cy^2$ be a primitive integral binary quadratic form with discriminant $D$ not a perfect square, positive definite if $D<0$, and let $A,B$ be as in Theorem~\ref{thm:maintheorem}. Let $\setS$ be a finite set of primes and $e_0\in \mathbb{N}$. 
Write $N_{bad}(N)$ for the number of primes $p<N$ such that 
\begin{enumerate}
\item $\frac{p-A}{B}$ is an integer with $\mu_{\setS}^2\left(\frac{p-A}{B}\right)=1$,
\item $p^{e_0}\nmid \frac{p-A}{B}$ for all $p\in \setS$, and 
\item $\frac{p-A}{B}$ is primitively represented by some form in the genus of $f$ but not by all forms in the genus of $f$.
\end{enumerate}
Then
$$N_{bad}(N)\ll \frac{N}{(\log N)^{3/2+\delta'}} ,$$
where $\delta'>0$ is a constant.
\end{theorem}

We provide a few remarks on the above generalization. Theorem~1.2 in \cite{bourgain2012representation} concerns squarefree shifted primes less than $N$ represented by all forms of the genus of a quadratic form of discriminant $D$, where $D$ is allowed to be a power of $\log N$. Since our forms in question have fixed discriminant that does not grow with $N$, the statement above can in fact be improved with $3/2$ replaced by $2$.  The key is that at some point in the proof of Theorem~1.2 in \cite{bourgain2012representation} one has to multiply by the class number which is roughly $\sqrt{D}$, and this affects the power of $\log N$ in the denominator above, but would not occur in our situation. Nonetheless, the result above is good enough for us, so we state it with the $(\log N)^{3/2+\delta}$.

As for the loosening of the squarefree condition to allow for bounded powers of primes from a fixed set of bad primes, we do not include the proof because it is, on the one hand, straightforward, but on the other hand does require re-writing most of the steps in \cite{bourgain2012representation}. To give an idea of what happens in these steps, the key input in \cite{bourgain2012representation} is Theorem~2.1 in that paper, which is a result in combinatorial group theory that has nothing to do with square free versus the slightly more general numbers we consider here. Theorem~2.1 there states that, given a finite abelian group $G$ and a subset $A$ of $G$, the sum set $s(A):=\{\sum x_i\;|\; x_i\in A \mbox{ distinct}\}$ is either the whole group $G$, or $A$ satisfies some specific unfortunate criteria (e.g. it is close to a subgroup of $G$). This result is then applied to the quadratic forms question by letting $G=\mathcal C^2$ where $\mathcal C$ is the class group of primitive positive definite binary forms of discriminant $D$. Given a squarefree number $n$ (which may be taken to be a shifted prime), the set $A$ is the set $\{C_j^2\}$ where $C_j$ and $C_j^{-1}$ are the classes representing the $j$-th prime factor of $n$. The task is then to give an upper bound on the number of $n<N$ for which $A$ falls into the unfortunate cases of Theorem~2.1 of \cite{bourgain2012representation}, because if $s(A)=G$ one has exactly that $n$ is represented by all forms of the genus if it is represented by one. 

Injecting multiples of square free numbers by bounded powers of a fixed set of primes into the analysis behind these upper bound changes them by a constant multiple depending on these bad primes and the powers they are allowed to appear with. This is much more straightforward than attempting to simply remove the squarefree condition with no limitations, in which case the analysis as it stands in \cite{bourgain2012representation} would not neatly generalize.

Theorem \ref{thm:maintheorem} is now a direct consequence of Proposition \ref{prop:lowerboundSsq} and Theorem \ref{thm:generalizationBourgainFuchs}.

\section{Local conditions}
\label{sec:local}
Here, we tackle a key issue in the computations in Section~\ref{sec:half-dimensional-sieve}: the existence of a table-admissible integer $d$ and an appropriate $L$ as described in Iwaniec's tables. Indeed, as we show below, in certain cases this is not true.

\begin{lemma}\label{lemlocal}
Assume that $(A,B)=1$. Then there exist values $d,L$ such that:
\begin{enumerate}
    \item $d$ is table-admissible and $L$ is residue-admissible for $R_f$ in the language of Section~\ref{sec:representability};
    \item $\mu^2(d)=1$;
    \item $2\mid BAd$; and
    \item $(BdL+A,QBd)=1$,
\end{enumerate}
unless we are in one of the following cases:
\begin{itemize}
\item $2\nmid AB$ and $D \equiv 5 \mod{8}$
\item $2\nmid AB$ and $\vartheta_2\geq 4$
\item $2\nmid AB$ and $\vartheta_2 =2$ and $D_2\equiv 1 \mod {4}$.
\item $3\mid D$, $3\nmid A$, $3\mid Ba+A$ and $\vartheta_3>1$.
\end{itemize}

Moreover, if we remove the squarefree condition $\mu^2(d)=1$, then the only exceptional case is $2\nmid AB$ and $D \equiv 5 \mod{8}$.  Barring this exceptional case, we can choose such a tuple such that $d$ is of the form $d=2^{\varepsilon_2}3^{\varepsilon_3}d'$ with $\mu^2(d')=1$, $\gcd(d',6)=1$ and $\varepsilon_2\in\{0,1,2,3,4\}$ and $\varepsilon_3\in\{0,1,2\}$.
\end{lemma}

\begin{proof}[Proof of Lemma \ref{lemlocal}]
We will choose $d$ to be composed only of powers of $2$ and powers of $3$. In particular, in the case $\mu^2(d)=1$ we will choose $d\in \{1,2,3,6\}$.

Note that if $q$ is a prime with $q\mid B$, then $q\nmid A$ and $(BdL+A,q)=1$ independent of the choice of $d,L$.\par
We consider now primes $p_i>3$ which divide $D$. Assume that $d$ is not divisible by $p_i$. Table~\ref{table1} shows that $\tau(\varepsilon_{p_i},\vartheta_{p_i})=p_i$ and that there are $\frac{p_i-1}{2}>1$ admissible residue classes of $L$ modulo $p_i$. We want that $(BdL+A,p_i)=1$. Note that the congruence $BdL+A\equiv 0 \mod{p_i}$ rules out zero congruence classes if $p_i\mid B$ (as then $p_i\nmid A$) and it rules out exactly one congruence class for $L$ if $p_i\nmid B$. We choose out of the $\frac{p_i-1}{2}$ admissible congruence classes for $L$ modulo $p_i$ one which has the property that $(BdL+A,p_i)=1$. This works for all values of $d$ which are not divisible by $p_i$.

It remains to consider the primes $2$ and $3$. We make a case distinction.

{\bf Considerations for the place $p_i=2$.}

\textit{Case 1:} We assume that $2\mid A$ (which implies $2\nmid B$). In this case we choose $d$ to be odd, i.e., $\varepsilon_2=0$. Then we certainly have $2\mid ABd$.

Next we consider Table~\ref{table2} for the prime $2$. If $\vartheta_2=0$, then $D\equiv 1 \mod 4$ as we always have $D \equiv 0,1\mod{4}$. In this case we have $\tau(\varepsilon_2,\vartheta_2)=1$ and $2\nmid Q$. If $\vartheta_2\geq 1$, then we automatically have $\vartheta_2\geq 2$. In the case $\vartheta_2=2$ and $D_2\equiv -1 \mod{4}$ we have $\tau(\varepsilon_2,\vartheta_2)=4$ and choose $L$ according to the table. Note that in this case $L$ is automatically odd. If $\vartheta_2=2$ and $D_2\equiv 1 \mod{4}$ we have $\tau(\varepsilon_2,\vartheta_2)=1$ and $Q$ is odd. In the case that $\vartheta_2\geq 3$, we have $\tau(\varepsilon_2,\vartheta_2)=4$ or $\tau(\varepsilon_2,\vartheta_2)=8$ and choose $L$ modulo $4$ or modulo $8$ according to the tables. Note that also in this case $L$ is automatically odd. To summarise our findings for the prime $2$, we observe that we can always find $L$ admissible at the prime $2$, and either $Q$ is odd or $L$ is odd. If $Q$ is even, then $(BdL+A,2)=(L,2)=1$ and if $Q$ is odd, then $QBd$ is odd.

\textit{Case 2:} We assume $2\mid B$. In this case again we choose $d$ odd and find some admissible $L$ as in Case 1. Note that now we automatically have $(BdL+A,2)=1$ independent of the choices for $d,L$.

\textit{Case 3:} $2\nmid AB$. In this case we have to choose $d$ even, i.e. $\varepsilon_2\geq 1$. In case we find an admissible $d,L$ with $d$ even, we automatically have $(BdL+A,2)=1$, i.e. we only need to decide if some admissible value exists. We can now read off the existence of an admissible tuple $d,L$ at the place $2$ from the Table~\ref{table2}.

In particular, if $\vartheta_2=0$ (i.e. $D$ is odd), then such a tuple exists if and only if $D \equiv 1 \mod 8$.

In the case $\vartheta_2>0$ and $\vartheta_2$ even, an admissible tuple $d,L$ always exists. Indeed in this case we can use one of the rows (\hyperref[row18]{18}), (\hyperref[row19]{19}), or (\hyperref[row20]{20}). 

If $\vartheta_2$ is odd, then we always have $\vartheta_2\geq 3$, and we can use row (\hyperref[row17]{17}).

If we want to add in the additional condition that $\mu^2(d)=1$, then we can only choose $\varepsilon_2 \in \{0,1\}$. This is no problem for Case 1 and Case 2 as were we may use $\varepsilon_2=0$. However, in Case 3 we must be in the case $\varepsilon_2=1$. An inspection of Table~\ref{table2} shows that this is only possible if we are in the case $\vartheta_2=2$ and $D_2\equiv -1 \mod{4}$ or if we are in the case $\vartheta_2=3$.

In the case that $2\nmid AB$, $\vartheta_2=2$ and $D_2\equiv 1 \mod{4}$, we choose $\varepsilon_2\in\{2,3\}$ depending on $D_2$ modulo $8$ as described in rows (\hyperref[row19]{19}) and (\hyperref[row20]{20}).

In the case $2\nmid AB$ and $\vartheta_2\geq 4$, we choose $\varepsilon_2$ as follows: If $4\leq \vartheta_2\leq 6$, then we use line (15) and (16) and choose $\varepsilon_2=\vartheta_2-3$ or $\varepsilon_2=\vartheta_2-2$ depending on the parity of $\varepsilon_2$. If $\vartheta_2\geq 7$, then we choose $\varepsilon_2=2$ as in line (\hyperref[row13]{13}).

{\bf Considerations for the place $p_i=3$.}

This is only relevant if $3\mid D$ as otherwise $3\nmid Qd$. We recall that if $3\mid D$ then $3\nmid a$.

\textit{Case 1:} $3\mid A$. In this case we have $3\nmid B$ and if we choose $d$ not divisible by $3$, then we have $(BdL+A,3)=1$. Such a choice is always possible by Table~\ref{table1}.

\textit{Case 2:} $3\mid B$. In this case $3\nmid A$ and $(BdL+A,3)=1$ for every admissible tuple $d,L$. The existence again follows from Table~\ref{table1}.  In particular, we can choose $\varepsilon_3=0$.

\textit{Case 3:} $3\nmid AB$ and $aB+A\not\equiv 0\mod{3}$: 
In this case we may choose $\varepsilon_3=0$, i.e. $d$ not divisible by $3$. Then we have $\tau(\varepsilon_{p_i},\vartheta_{p_i})=p_i=3$ and there is exactly one admissible congruence class for $L$ modulo $3$, which is the congruence class $L\equiv ad \mod{3}$ (see row (\hyperref[row1]{1}) in the table). In this case we have
$$(BdL+A,3)=(Bad^2+A,3)=(Ba+A,3)=1.$$

\textit{Case 4:} $3\nmid AB$ and $3\mid aB+A$. We have learned from Case 3, that if $\varepsilon_3=0$, then we must have $L\equiv ad \mod{3}$ and hence $(BdL+A,3)=(Bad^2+A,3)=(Ba+A,3)=3$ which is not allowed. Hence in this case we must have $\varepsilon_3>0$. If we have $\varepsilon_3>0$, i.e., $3\mid d$, then we automatically have $(BdL+A,3)=(A,3)=1$. I.e. the question reduces to when an admissible tuple $d,L$ at the place $3$ exists. If we are allowed to choose $\varepsilon_3$ freely, then Table~\ref{table1} shows that some admissible tuple always exists (for example by setting $\varepsilon_3=\vartheta_3$ in rows (\hyperref[row2]{2}), (\hyperref[row3]{3}) or (\hyperref[row4]{4})). However, if we additionally want to achieve $\mu^2(d)=1$, then the only possibility for $\varepsilon_3$ is $\varepsilon_3=1$. Then Table~\ref{table1} shows that the only case in which there is an admissible choice for $L$ is in the case $\vartheta_3=\varepsilon_3=1$.

For the more general case, we again assume for a moment that $3\nmid AB$, $3\mid Ba+A$ and $\vartheta_3 > 1$. If $\vartheta_3=2$, then we choose $\varepsilon_3=2$, which is admissible by rows (\hyperref[row2]{2}) and (\hyperref[row3]{3}). If $\vartheta_3\geq 3$, then we choose $\varepsilon_3=2$ which is admissible by row (\hyperref[row1]{1}).
\end{proof}

\newpage

\bibliographystyle{alpha}
\bibliography{refs}

\newcommand{\etalchar}[1]{$^{#1}$}
\begin{thebibliography}{GLM{\etalchar{+}}03}

\bibitem[BF11]{BourgainFuchs}
Jean Bourgain and Elena Fuchs.
\newblock A proof of the positive density conjecture for integer {A}pollonian
  circle packings.
\newblock {\em J. Amer. Math. Soc.}, 24(4):945--967, 2011.

\bibitem[BF12]{bourgain2012representation}
Jean Bourgain and Elena Fuchs.
\newblock On representation of integers by binary quadratic forms.
\newblock {\em International Mathematics Research Notices},
  2012(24):5505--5553, 2012.

\bibitem[FFH{\etalchar{+}}24]{primecomponents}
Holley Friedlander, Elena Fuchs, Piper Harris, Catherine Hsu, James Rickards,
  Katherine Sanden, Damaris Schindler, and Katherine~E. Stange.
\newblock Prime and thickened prime components in apollonian circle packings.
\newblock \url{https://arXiv.org/abs/2410.00177}, 2024.

\bibitem[Fuc13]{FuchsBulletin}
Elena Fuchs.
\newblock Counting problems in {A}pollonian packings.
\newblock {\em Bull. Amer. Math. Soc. (N.S.)}, 50(2):229--266, 2013.

\bibitem[GLM{\etalchar{+}}03]{GLMWY}
R.~L. Graham, J.~C. Lagarias, C.~L. Mallows, A.~R. Wilks, and C.~H. Yan.
\newblock Apollonian circle packings: number theory.
\newblock {\em J. Number Theory}, 100(1):1--45, 2003.

\bibitem[Iwa72a]{Iwaniec72shiftedprimes}
H.~Iwaniec.
\newblock Primes of the type {$\phi(x,\,y)+A$} where {$\phi$} is a quadratic
  form.
\newblock {\em Acta Arith.}, 21:203--234, 1972.

\bibitem[Iwa72b]{Iwaniec72a}
H.~Iwaniec.
\newblock Primes represented by quadratic polynomials in two variables.
\newblock {\em Bull. Acad. Polon. Sci. S\'er. Sci. Math. Astronom. Phys.},
  20:195--202, 1972.

\bibitem[Iwa74]{Iwaniec74quadraticfunctions}
H.~Iwaniec.
\newblock Primes represented by quadratic polynomials in two variables.
\newblock {\em Acta Arith.}, 24:435--459, 1973/74.

\bibitem[Ple67]{Pleasants}
P.~A.~B. Pleasants.
\newblock The representation of primes by quadratic and cubic polynomials.
\newblock {\em Acta Arith.}, 12:131--163, 1966/67.

\bibitem[Ric25]{GHBQFRepCheck}
James Rickards.
\newblock B{QFR}ep{C}heck.
\newblock \url{https://github.com/JamesRickards-Canada/BQFRepCheck}, 2025.

\bibitem[Sar07]{Sarnak}
Peter Sarnak.
\newblock Letter to {J}. {L}agarias.
\newblock \url{https://web.math.princeton.edu/sarnak/AppolonianPackings.pdf},
  2007.

\bibitem[{The}25]{PARI2}
{The PARI~Group}, Univ. Bordeaux.
\newblock {\em {PARI/GP version \texttt{2.18.0}}}, 2025.
\newblock available from \url{http://pari.math.u-bordeaux.fr/}.

\end{thebibliography}

\end{document}